\documentclass[12pt]{amsart}
\usepackage[american]{babel}
\usepackage{amsmath}
\usepackage{latexsym}
\usepackage{amssymb}
\usepackage[shortlabels]{enumitem}
\usepackage{tikz}
\usetikzlibrary{matrix,arrows,decorations.pathmorphing}
\usepackage{tikz-cd}
\usepackage[noadjust]{marginnote}
\usepackage{mathtools,url}
\usepackage{comment}
\DeclareMathAlphabet\mathbfcal{OMS}{cmsy}{b}{n}
\theoremstyle{plain}
\newtheorem{Thm}[subsection]{Theorem}
\newtheorem{Cor}[subsection]{Corollary}
\newtheorem{Prop}[subsection]{Proposition}
\newtheorem{Lem}[subsection]{Lemma}

\theoremstyle{definition}
\newtheorem{Construction}[subsection]{Construction}
\newtheorem{Ex}[subsection]{Example}
\newtheorem{Rem}[subsection]{Remark}
\newtheorem{Def}[subsection]{Definition}

\newtheorem{Caveat}[subsection]{Caveat}

\DeclareMathOperator{\Hom}{Hom}

\DeclareSymbolFont{yhlargesymbols}{OMX}{yhex}{m}{n} \DeclareMathAccent{\yhwidehat}{\mathord}{yhlargesymbols}{"62}

\renewcommand{\phi}{\varphi}
\newcommand{\RR}{\mathbb{R}}

\newcommand{\ZZ}{\mathbb{Z}}
\newcommand{\NN}{\mathbb{N}}
\newcommand{\TT}{\mathbb{T}}
\newcommand{\id}{\mathrm{id}}

\newcommand{\supp}{\mathrm{supp}}

\renewcommand{\emptyset}{\varnothing}
\renewcommand{\setminus}{-}

\newcommand{\lca}{\mathsf{lca}}
\newcommand{\apl}{\mathsf{apl}}
\newcommand{\tg}{\mathsf{tg}}
\newcommand{\tab}{\mathsf{tab}}

\newcommand{\ab}{\mathsf{ab}}
\newcommand{\kab}{\mathsf{kab}}
\newcommand{\ktg}{\mathsf{ktg}}

\newcommand{\cT}{\mathcal T}
\newcommand{\cS}{\mathcal S}

\newcommand{\doublewidehat}[1]{\widehat{\widehat{#1\, }}}

\usepackage[colorinlistoftodos,prependcaption,textsize=tiny]{todonotes}

\title{\bf{Some notes on Pontryagin Duality of abelian topological groups}}
\author{Karl Heinrich Hofmann and Linus Kramer}

\begin{document}

\begin{abstract}
A k-group is a topological group $K$ with the property that every algebraic homomorphism
  $f:K\to H$ is continuous, provided that $f$ is continuous on every compact
  subset $C\subseteq K$. We show that for each topological group $G$ there is a topological
  k-group $kG$ arising functorially from $G$ by refining the topology.
  For each abelian pro-Lie k-group $G$ we show that
 the evaluation morphism $\eta_G\colon G\to \doublewidehat{G}$ is an
 isomorphism. An example due to Leptin shows that $\eta_G$ need not be continuous
 if the pro-Lie group $G$ is not a k-group.
 However, a result due to Aussenhofer shows that $\eta_G$ is then an open and bijective map.
 We observe that the openness of $\eta_G$ follows also from a categorical viewpoint.
\end{abstract}


\maketitle

Some observations on
topological abelian groups  and their duality appear to
be useful to recall, notably those with 
an emphasis on commutative pro-Lie groups. Some of
our comments appear to be new and some of the items we
note here have been known for some time, yet appear to
be worthwhile to be recalled. A good deal of the 
necessary  background
is  available in sources like \cite{hofi} and \cite{hofii}.

The circle group $\RR/\ZZ=\TT$ (written additively), for each
abelian topological group $G$, gives rise to its 
{\it character group } \[\widehat G=\Hom(G,\RR/\ZZ),\] which is again
a topological abelian group when we endow it with the
topology of uniform convergence on compact subsets of $G$.
The repetition of this first step leads to the creation of
the bidual $\doublewidehat{G}=\Hom(\Hom(G,\RR/\ZZ),\RR/\ZZ)$, and if $g\in G$
and $\chi\in\widehat G$ yield the element $\chi(g)\in\RR/\ZZ$, then
the following {\it evaluation homomorphism}
is immediately present:
\[\eta_G\colon G\to \doublewidehat{G}\text{ defined by }
\eta_G(g)(\chi)=\chi(g).\]
A central portion of the classical theory of \textsc{Pontryagin Duality}
is the statement that

\smallskip\noindent
(A) {\it $\eta_G$ is an isomorphism (algebraically and
topologically) whenever $G$ is {\it locally compact}.}

\smallskip
Yet many items in the literature lead us beyond these limitations.
Another significant
aspect of that classical duality is that

\smallskip \noindent
(B) {\it the category of locally compact abelian groups $\lca$ is closed
under the  passage from $G$ to $\widehat G$.}

\smallskip
While significant progress about duality of abelian topological groups
beyond local compactness was indeed achieved (as documented in the recent
monograph \cite{adgb}
by \textsc{Aussenhofer}, \textsc{Dikranjan} and \textsc{Giordano Bruno}),
 the overall picture is neither clear nor complete. 
As a test of this claim we address here what is known
for the category of \emph{abelian
pro-Lie groups} $\apl$ in the regard of duality.
This category, which contains
all $\lca$ groups, is in contrast to
the category of $\lca$ groups complete, that is, closed under the
formation of all limits. 
It contains the category of all weakly complete real vector spaces,
being isomorphic to $\RR^I$ (for an arbitrary set $I$) with the
product topology. If $G=\RR^I$, then $\widehat G\cong\RR^{(I)}$ (the
direct sum of cardinality $I$-many copies of $\RR$ with the finest
locally convex vector space topology) does not belong to the category
of pro-Lie groups. This example illustrates that, even on an
elementary level, \emph{the category of
abelian pro-Lie groups $\apl$ differs significantly from its dual category}. 
Even though some 
information is provided in Chapter 4 of \cite{hofii}, the status
of information on duality of $\apl$ groups is far from satisfactory.
This encourages us  to offer in this note
some complementary pieces of information
on that duality.

It will be useful 
to discuss the details of a class of  examples of
prodiscrete
(hence pro-Lie) groups $E$  for which $\eta_E$ fails to be continuous,
further
$\widehat E$ fails to be complete, while its bidual $\doublewidehat{E}$ is discrete.
These examples,  in essence due to \textsc{Leptin} (\cite{lep} 1955), later
mentioned by 
\textsc{Noble} (\cite{nobthesis} Chapter 1, 1967, and \cite{nob}, Example 1.6, 1970), \textsc{Banaszczyk} (\cite{ban} 1991) 
also illustrate certain aspects of the general progress
that was contributed to duality by \textsc{Aussenhofer} (\cite{aus}, 1999).

We also point out 
that the environment of pro-Lie groups
permits a  different access to 
\textsc{Aussenhofer}'s insight that for
large classes of abelian topological groups $G$ including
$\apl$ groups  the evaluation morphism
$\eta_G$ \emph{is bijective and open}, in other words, that
$\eta_G^{-1}$ exists and is continuous. 
We explain in which
sense duality is necessary and sufficient for the continuity of
$\eta_G$ for pro-Lie groups. 

While  $\eta_G$ may be discontinuous even for $\apl$ groups, as \textsc{Leptin}'s
example had illustrated for 70 years, it is also known that 

\smallskip\noindent
(C) \emph{for each compact subspace $C$ of $G$, the restriction
$\eta_G|C\colon C\to \doublewidehat{G}$ is continuous,}

\smallskip\noindent
as is pointed out in 
\cite{adgb},
Proposition 13.4.1, in
the proof of
\cite{hofi}, Theorem 7.7 (iii) or in \cite{strop}, proof of Lemma 2.10. 
This points into the direction of what became known as \emph{k-groups}.
According to \textsc{Noble} \cite{nobthesis,nob}, a group homomorphism defined on a 
topological group $G$
is called \emph{k-continuous} if each restriction to a compact subset
of $G$ is continuous.
Accordingly, $\eta_G$ \emph{is k-continuous}.
\textsc{Noble} calls a topological
group $G$ a \emph{k-group} if each k-continuous
homomorphism from $G$ to a topological
group is  continuous.
In this terminology,  for every  k-group $G$, the evaluation morphism
$\eta_G$ is continuous. 
This justifies  
our review of the validity
of duality in the context of topological abelian k-groups 
(see \cite{nobthesis,nob,nobii}). Indeed we present in Section~1
a new category
theoretical aspect, namely, that \emph{the subcategory of k-groups
is coreflexive in the category of topological groups} and  what this
 means in concrete circumstances.  
 All locally compact groups are k-groups. The examples $E$ of pro-Lie groups
 we shall discuss in Section 2 fail to be  k-groups. 

 A brief summary of our results then reads as follows:
 
\smallskip\noindent    

\textbf{Theorem.} 
{\em 
 {\bf (A)} For each topological group $G$ there is a topological
  k-group $kG$ arising functorially from $G$ by refining the topology
  in such a fashion that the natural bijection $\kappa_G\colon kG\to G$
  preserves all compact subspaces of $G$ and that each morphism
  $f\colon H\to G$ from a k-group $H$ to $G$ factors via $f'\colon H\to kG$
  through $kG$ such that $f=\kappa_G\circ f'$.
  
 {\bf (B)} For each abelian pro-Lie k-group $G$
 the evaluation morphism $\eta_G\colon G\to \doublewidehat{G}$ is an
 isomorphism.} 

\smallskip\noindent
One might say that an abelian pro-Lie group $G$ 
\emph{satisfies \textsc{Pontryagin Duality}
if it is a k-group,}
and that \emph{no abelian pro-Lie group is ever far away from a k-group.}
Indeed we conclude our discourse with the open question whether
for any abelian pro-Lie group $G$ the double dual
$\doublewidehat G$ is automatically a k-group.

\subsection*{Conventions}
In what follows, all topological spaces are assumed to be {\em Hausdorff}
unless stated otherwise. A {\em morphism} between topological groups is a
continuous group homomorphism. A {\em Lie group} is a topological group
whose underlying space is a smooth manifold (locally diffeomorphic
to some $\RR^m$) 
such that the map $(x,y)\mapsto xy^{-1}$ is smooth.
We impose {\em no countability conditions}
on Lie groups. A {\em pro-Lie group}
is a topological group $G$ satisfying the following equivalent conditions.
\begin{enumerate}[\rm(i)]
 \item $G$ is isomorphic to a closed subgroup
of a product of Lie groups.
\item $G$ is complete and $G$ is (in the category of topological groups) a
projective limit of Lie groups \cite[Thm.~2.37]{hofii}.
\end{enumerate}

\subsection*{Acknowledgement}
First of all, we thank the referee for very helpful and constructive 
remarks and suggestions on an earlier version of the present article.
Our research project was carried out during a stay in Oberwolfach
as Oberwolfach Research Fellows in 2025.
We thank the Institute for its hospitality.
The 2nd author was supported by the DFG under Germany's Excellence Strategy 
EXC 2044/2--390685587, Mathematics Münster: Dynamics-Geometry-Structure.

\section{k-groups}

We recall that a \emph{k-space} (sometimes called a \emph{Kelley space})
is a Hausdorff space $X$ with the following
property: a map $f\colon X\to Y$ to any other topological space $Y$ is continuous if and
only if the restriction of $f$ to every compact subspace $C\subseteq X$ is continuous.
These spaces were introduced by \textsc{Hurewicz}; 
they play a major role in algebraic
topology because they have many favorable properties. Every locally compact space
and every first countable space is a k-space. We refer to  \cite{dug} XI.9 or \cite{eng} 3.3
for more results about these spaces. If $X$ is any Hausdorff space, with topology
$\cT$, then there is a finer topology $\cT_{\max}\supseteq\cT$ having the same compact sets as $\cT$, such that $(X,\cT_{\max})$
is a k-space. The topology $\cT_{\max}$ has an explicit description as follows.
A subset $A\subseteq X$
is $\cT_{\max}$-closed if and only if $A\cap C$ is closed, for every compact subset $C\subseteq X$, that is,
\[
\cT_{\max}=\{U\subseteq X : (X\setminus U)\cap C\text { is closed for every
compact }C\subseteq X\}.
\]
It follows that 
$\cT_{\max}$ is the unique largest topology containing $\cT$ that has the same compact sets
as $\cT$. The assignment $(X,\cT)\mapsto(X,\cT_{\max})$ is a functor which is right adjoint
to the inclusion functor of k-spaces into Hausdorff spaces. These categorical
aspects are discussed in \cite{steenrod} and in \cite{maclane}, VII.8.
The category of k-spaces also has some drawbacks. Notably, a product of k-spaces need not be a k-space  in the product topology, see 
Remark~\ref{kgroupremark} below. Also, $\cT_{\max}$ need not be a group topology if $\cT$ is a group topology.
Nevertheless, one may study the group objects
in the category of k-spaces. This is carried out in \textsc{LaMartin}'s work~\cite{lam}.

\textsc{Noble} \cite{nobthesis,nob} introduced a variation of this construction in the context
of topological groups in his 1967 PhD thesis and considered k-groups. However, he did not discuss the
categorical aspects of his construction. In what follows, we give a self-contained
introduction to k-groups from a categorical viewpoint.

The following general lemma is certainly well-known. 
\begin{Lem}\label{supremum}
Let $G$ be a group and let $\tau$ be a set of group topologies on $G$
(which need not be Hausdorff).
Then $\tau$ has a unique supremum $\cT=\sup\tau$ in the partially ordered set of all topologies
on $G$, and $\cT$ is a group topology.
\end{Lem}
\begin{proof}
    For each $\cS\in\tau$ let $G_{\cS}$ denote the topological group $G$ with topology
    $\cS$. We put $H=\prod_{\cS\in\tau}G_{\cS}$ and we consider the diagonal
    map $d\colon G\to H$, and  $\cT=\{d^{-1}(U) : U\text{ open in }H\}$.
    Then $\cT$ is a group topology and an upper bound for $\tau$, because
    each map $p_{\cS}\circ d\colon G\to G_{\cS}$ is continuous and
    bijective.
    If some topology $\cT'$ on $G$ is an upper bound for $\tau$, then
    each map $p_{\cS}\circ d$ is $\cT'$-continuous and hence $d$ is $\cT'$-continuous.
    Thus $\cT'\supseteq\cT$. This shows that the group topology $\cT$ is a least upper bound for $\tau$
    in the set of all topologies on $G$.
\end{proof}
The following notion is due to \textsc{Noble}~\cite{nobthesis,nob}.
\begin{Def}\label{kgroupdef}
    We call a homomorphism $f\colon G\to H$ between topological groups $G,H$
    \emph{k-continuous} or a \emph{k-morphism}
    if the restriction of $f$ to any compact subset
    $C\subseteq G$ is a continuous map. 
    A k-continuous homomorphism is thus sequentially continuous.
    We call $G$ a \emph{k-group} if every k-continuous
    homomorphism $f\colon G\to H$
    is continuous, for every topological group $H$.
  \end{Def}
  
    For example,
every locally compact group and every first countable topological group
    is a k-group. More generally, a topological group whose underlying
    topology     is a k-space is a k-group. 

We denote the full subcategory of k-groups in the category
of topological groups $\tg$ by $\ktg$.

\begin{Construction}\label{kgroup}
Let $G$ be a topological group, with group topology $\cT$. By Lemma
\ref{supremum}, the set $\tau$ of all group topologies on $G$ which refine $\cT$ and which have the
same compact sets as $\cT$ has a unique supremum $k\cT=\sup\tau$, and 
$k\cT$ is a group topology. Since $\cT_{\max}$ is an upper bound for $\tau$,
the group topology $k\cT$ has the same compact sets as $\cT$,
that is, $k\cT=\max\tau$.
We denote the resulting topological group by $kG$ for short.
\end{Construction}
We will see below in \ref{kgroupremark}
that $k\cT$ may be strictly smaller than $\cT_{\max}$.
Note that our definition of the group topology $k\cT$ does not give an explicit
description of this topology, in contrast to the topology $\cT_{\max}$ introduced
above.

\begin{Lem}
    The group $kG$ is a k-group, and $kG$ has the same compact subsets as $G$.
\end{Lem}
\begin{proof}
    Let $\cT$ denote the group topology of $G$.
    By Construction~\ref{kgroup}, the topological group $kG$ has the same compact subsets as $G$.
    Let $f\colon kG\to H$ be a k-continuous group homomorphism.
    Since $kG$ and $G$ have the same compact subsets, $f$ is also k-continuous
    as a map from $G$ to $H$. If $B\subseteq H$ is a closed subset, then 
    $f^{-1}(B)\cap C$ is closed for every compact subset $C\subseteq G$, because $f$ is k-continuous.
    The (possibly non-Hausdorff) group topology $\cS=\{f^{-1}(U) : U\subseteq H\text{ open}\}$ on $G$ is therefore contained in the topology $\cT_{\max}$. Hence $\cT'=\sup\{\cT,\cS\}\subseteq\cT_{\max}$
    is a Hausdorff group topology having the same compact sets as $\cT$,
    and  $f$ is continuous with respect to $\cT'$. But then $f$ is also continuous
    with respect to $k\cT\supseteq\cT'$.
\end{proof}
The identity map on the underlying group $G$ is a morphism $\kappa_G\colon kG\to G$, which plays
a special role.
\begin{Prop}\label{adjoint}
 Let $G$ be a topological group. Then $\kappa_G\colon kG\to G$ has the following universal property.
 If $H$ is a k-group, and if $f\colon H\to G$ is a morphism, then $f$ factors uniquely as $f=\kappa_G\circ f'$,
 \[
  \begin{tikzcd}
   H \arrow{dr}{f}\arrow{d}[swap]{f'} \\
   kG\arrow{r}{\kappa_G} & G .
  \end{tikzcd}
 \]
 In particular, $G$ is a k-group if and only if $G=kG$.
\end{Prop}
\begin{proof}
 Since $kG$ and $G$ have the same compact subspaces, $f$ is k-continuous as a map from $H$ to $kG$.
 Hence $f=f'\colon H\to kG$ is a morphism. Since $\kappa_G$ is the identity map on the underlying group
 $G$, the map $f'$ is uniquely determined by $f$.
\end{proof}

Let us now consider the  category $\tg$ of topological groups
(and continuous group morphisms), and the full subcategory $\ktg$
of k-groups and continuous group morphisms, with the  
inclusion functor $\iota\colon \ktg\to\tg$.
We shall now use standard category theoretical notation as is presented
e.g.~in
\cite{hofi}  (Theorem A3.28 ff., in \cite{hofi}, p.~814ff.). 
The machinery of adjoint functors (see e.g. \cite{hofi}, Definition A3.29ff.,
see also \cite{hofii}, A1.40 and A1.41) shows the following from Proposition~\ref{adjoint}.
\begin{Thm}
     The inclusion functor $\iota\colon \ktg\to\tg$
has a right adjoint $k\colon\tg\to\ktg$ that maps a group $G$ to the k-group
$kG$.
\end{Thm}
It may be useful to repeat explicitly what
this adjunction of functors means.
For each topological group $G$ there is a k-group $kG$ and a natural
morphism $\kappa_G\colon \iota kG \to G$ of topological groups such that
for every k-group $H$ and each morphism
$f\colon \iota H\to G$ of topological groups there is a unique
k-morphism $f'\colon  H\to kG$
such that $f=\kappa_G\circ \iota f'$ and that
$$f\mapsto f'\colon \tg(\iota H,G)\to \ktg(H,kG)$$
is a natural bijection.

Repeated again in other words: 
for any  topological group $G$ we obtain functorially
a k-group $kG$, 
 and a morphism $\kappa_G\colon \iota kG\to G$ of
topological groups (in fact turning out to be \emph{bijective}).
Then the universal property
explained above is summarized in  the following diagram:
\begin{center}
    \begin{tabular}{ccc}
        $\tg$  && $\ktg$ \\ \hline
        \begin{tikzcd}
    G && \arrow{ll}[swap]{\kappa_G} \iota kG \\
     && \iota H \arrow{u}[swap]{\iota f'}\arrow{llu}{f}
     \end{tikzcd}
     &&
     \begin{tikzcd}
         kG \\ \arrow{u}{\exists!}[swap]{f'} H
     \end{tikzcd}
    \end{tabular}
\end{center}
The diagram might be expressed briefly by saying
(as in Proposition~\ref{adjoint}) that \emph{any morphism
of topological groups from any k-group $H$ into the topological group
$G$  factors through $\kappa_G$.}

\begin{Rem}\label{coreflexive}
The subcategory $\iota\colon \ktg\to\tg$ is what is called a \emph{coreflective} subcategory,
with coreflector $k\colon \tg\to\ktg$. Such a subcategory has several important properties which we now recall.
Suppose that $D\colon J\to\ktg$ is any small diagram. 
\begin{enumerate}[\rm(1),nosep]
    \item If $\iota D\colon J\to \tg$ has a colimit
$K$, then $K\cong \iota K'$ for some $K'$ in $\ktg$---every colimit of k-groups
that exists in $\tg$
already exists in $\ktg$. (For the proof one puts $K'=kK$.
From the colimit property of $K$, there is a universal
morphism $K\to\iota K'$ and from the adjunction there is a universal
morphism $\iota K'=\iota kK\to K$.)
\item
In particular, if $N\unlhd G$ is a closed normal subgroup in a k-group $G$,
then $G/N$ is a k-group with respect to the quotient topology.
\item If $\iota D\colon J\to\tg$ has a limit $L$ in $\tg$, then 
  $kL$ is the limit of $D$ in $\ktg$. 
  (This holds because the right adjoint $k$ 
  preserves limits.)
\item
If $G_i$, for $i\in I$, is a family of k-groups, then 
$k(\prod_{i\in I}G_i)$ is the categorical product of the k-groups $G_i$ in 
the category $\ktg$. (This will be improved in Theorem~\ref{NobleTheorem} below.)
\item By (1) and (3), the category $\ktg$ is complete and cocomplete, since
$\tg$ is complete and cocomplete.
\end{enumerate}

Analogous remarks apply to the category $\tab$ of topological abelian groups
and the full subcategory of abelian k-groups $\kab$.
The inclusion $\kab\to\tab$ also has the right adjoint $k$. It follows in this case
that a finite product of abelian k-groups (which is also coproduct in
$\tab$) is a k-group.
\end{Rem}
In view of Remark \ref{coreflexive}(4),
the following result by \textsc{Noble}~\cite{nobthesis,nobii}
is rather surprising. We present a simplified
version of the proof \footnote{The author L.K.~did not understand
the (different) proof presented in \cite{nobii}.}
 given in \cite{nobthesis}.

 \begin{Thm}\label{NobleTheorem}
   Products of k-groups are again k-groups
   with respect to the product topology.
\end{Thm}
\begin{proof}
  Let $(G_i)_{i\in I}$ be a family of k-groups,
  with product $G=\prod_{i\in I}G_i$,
 with projections $p_i\colon G\to G_i$ and with product topology $\cT$. 
 We need to show that $\cT=k\cT$. 
 For $g\in G$ we put 
 $\supp(g)=\{i\in I : p_i(g)\neq e_i\}$,
 and $G'=\{g\in G : \supp(g)\text{ is countable}\}$.
 For $J\subseteq I$ we put $G_J'=\{g\in G': \supp(g)\cap J=\emptyset\}$.
 
 \smallskip\noindent
 \emph{Claim 1. For every $k\cT$-identity neighborhood $V\subseteq G$, there 
 is a finite set $J\subseteq I$ such that 
 $G'_J\subseteq V$.}
 
 \smallskip\noindent
 Assume that the claim is false. We pick an element $g_0\in G'\setminus V$,
 and enumerate the countable set $\supp(g_0)$ by a surjective map
 $\NN\to\supp(g_0)$. 
 Inductively, we choose $g_{n+1}\in G'\setminus V$ in such a way
 that $\supp(g_{n+1})$ contains none of first $n$ elements in each of the sets
 in $\supp(g_0),\supp(g_1),\ldots,\supp(g_n)$,
 and we fix a surjective map $\NN\to\supp(g_{n+1})$.
 Each $g_n$ has its support contained
 in the countable set $\bigcup_{m\in\NN}\supp(g_m)=K\subseteq I$. Moreover,
 for each $k\in K$ there exists some $m\in\NN$
 such that $p_k(g_n)=e_k$ holds for
 all $n\geq m$. Therefore the sequence $(g_n)_{n\in\NN}$ converges
 in the product topology $\cT$ to the identity element $e$.
 The set $C=\{g_n\mid n\in\NN\}\cup\{e\}\subseteq G$
 is thus compact (in $\cT$, and hence also in $k\cT$) with $V\cap C=\{e\}$. 
 This is a contradiction to the fact
 that $V$ is a $\cT_{\max}$-neighborhood of $e$,
 because $e$ is not isolated in $C$.
 Hence there is some finite set $J\subseteq I$ of cardinality $m$,
 with $G'_J\subseteq V$, and Claim~1 is true.
 
 \smallskip\noindent
 \emph{Claim 2. Let $J\subseteq I$ be a subset.
   The $k\cT$-closure of $G'_J$ contains
   $\prod_{j\in J}\{e_j\}\times\prod_{i\in I\setminus J}G_i$.}
 
 \smallskip\noindent
 Let
$g=(g_j)_{j\in J}\in \prod_{j\in J}\{e_j\}\times\prod_{i\in I\setminus J}G_i$
an consider the $\cT$-compact set
\[D=\prod_{j\in J}\{e_j\}\times\prod_{i\in I\setminus J}\{e_i,g_i\},\]
which contains $g$. Then $G'_J\cap D$ is dense in $D$.
Since $D$ is also $k\cT$-compact, $g$ is in the $k\cT$-closure of $G'_J$,
and Claim~2 is true.

\smallskip
Now we prove the theorem.
Let $U\subseteq G$ be any $k\cT$-identity neighborhood,
and let $V\subseteq U$ be a $k\cT$-identity neighborhood
with $VVV\subseteq U$.
By Claim 1, there is a finite subset $J\subseteq I$ with $G'_J\subseteq V$, of
cardinality $m$. Let $W\subseteq G$ be a $k\cT$-identity neighborhood
with $W^{\cdot m}=W\cdots W{\ (m\ \text{times})}\subseteq V$.
For each $j\in J$ we
we have the inclusion morphism $\iota_j\colon G_j\to G$.
We choose identity neighborhoods
 $W_j\subseteq G_j$ such that $\iota_j(W_j)\subseteq W$.
 Then $Z=\prod_{j\in J}W_j\times\prod_{i\in I\setminus J}G_i$
 is a $\cT$-identity neighborhood which is contained in the $k\cT$-closure 
 of $W^{\cdot m}G_J'$, and $W^{\cdot m}G_J'\subseteq VV$. The
 $k\cT$-closure of $VV$ is contained in $VVV$, whence
 \[
  Z\subseteq VVV\subseteq U,
 \]
 which shows that $U$ is a $\cT$-identity neighborhood.
Thus $\cT\supseteq k\cT\supseteq \cT$.
\end{proof}

\begin{Caveat}
    We noticed in Remark~\ref{coreflexive} that the category of 
    $k$-groups forms a full and coreflexive 
    subcategory
    $\iota\colon \ktg\to\tg$ of the category of topological groups.
    The coreflector is the `k-fication' functor $G\mapsto kG$.
    The category $\ktg$ contains all metrizable groups and 
    all locally compact groups (in fact, it contains
    all \v Cech-complete groups by \cite{eng} 3.9.5).
    Being coreflexive, the category $\ktg$ is complete and cocomplete.

    The inclusion functor $\iota$ preserves colimits,
    since it is left adjoint to $k$.
    It does \emph{not} preserve limits, as we shall see in the next section.
    Therefore, it is rather surprising that $\iota$ preserves products
    by Theorem~\ref{NobleTheorem},  which are after all special limits.
\end{Caveat}

However, the following is true
\footnote{The proofs presented in \cite{nobthesis} Cor.~5.3
or in \cite{nob} Cor.~1.8 appear to be incomplete.
Both proofs assume tacitly that every group topology on a subgroup
$H\subseteq G$ extends
to a group topology on $G$, which may well fail
if $H$ is not central in~$G$.}.

\begin{Thm}\label{opensubgroups}
Let $H\subseteq G$ be an open subgroup of the Hausdorff group $G$.
Then $G$ is a k-group if $H$ is a k-group. Conversely, if $H$ is
a central open subgroup and if $G$ is a k-group, then $H$ is a k-group.
\end{Thm}

\begin{proof}
  Let $\cT$ denote the topology on $G$.
 If the open subgroup $H$ is a k-group, then $H\to kG$
 is continuous. Since $H$ is open in $G$, this implies that $G\to kG$
 is continuous, whence $\cT=k\cT$.

Conversely, let $G$ be a k-group
 and let $H\subseteq G$ be a central open subgroup, with subspace
 topology $\cS$.
The topology of $kH$, which refines the subspace topology of $H$,
extends in a unique way to group topology  $\cT'$ on $G$
refining the topology
$\cT$ such that $kH$ is an open subgroup of $G$ in $\cT'$.
Here we  use that $H$ is central in $G$.
 
 If $C\subseteq G$ is a $\cT$-compact subset, then $C\cap gH$ is
 $\cT'$-compact for every coset $gH$. Moreover, $C$ intersects
 only finitely many such cosets nontrivially. Thus $C$ is also
 $\cT'$-compact. This shows that $\cT'\supseteq\cT$ is a group
 topology having the same compact sets as $\cT$, whence
 $\cT'=\cT$ and $H=kH$. 
\end{proof}

\begin{Rem}
    One might also consider the category $\mathrm{k-}\tg$
    whose objects are topological groups, and whose morphisms
    are $k$-continuous homomorphisms. In this category $\mathrm{k-}\tg$,
    whose hom-sets are possibly larger than those of $\tg$,
    the natural transformation $\kappa_G\colon kG\to G$ becomes invertible
    and gives an equivalence of categories
    \[\mathrm{k-}\tg\simeq\ktg.\]
\end{Rem}

\section{The Leptin-Noble-Banaszczyk Example} 

Before we illustrate how limits fail to be preserved
by $\iota\colon \kab\to\tab$
we have to inspect in detail a classical example described by Leptin
in 1955 \cite{lep} (see also \cite{nobthesis} from 1967, \cite{nob} from 1970,
and \cite{ban} from 1991).
Because of the significance of the example, we provide
some explicit reference at this point.

\begin{Def}\label{1.1}  Let $I$ denote the set of all ordinals
  called $\alpha$, $\beta$, $\gamma$ etc. less than the first
  uncountable ordinal $\omega_1$. 
  We fix a family of nontrivial discrete abelian groups $(Z_\alpha,+)$, for
 $\alpha\in I$. For example, $Z_\alpha$ may be a fixed finite or
 infinite cyclic group. Let $\widehat{Z_\alpha}=\Hom(Z_\alpha,\TT)$
 denote the dual of $Z_\alpha$, a compact abelian group.
 A key feature of $I$ is that for every countable subset
 $\{\alpha_n:n\in\NN\}\subseteq I$ there is some $\gamma\in I$
 such that $\alpha_n<\gamma$ holds for all $n$.
 This is true since the least  upper bound of
  a countable set of countable ordinals is again a countable ordinal.
  We put
  $Z=\bigoplus_{\alpha\in I} Z_\alpha$.
For the elements of $Z$ write $z=\sum_{\alpha\in I}z_\alpha$.
\end{Def}
Let $H_\alpha$ denote the subgroup
of $Z$ of all elements $z=\sum_{\beta\in I}z_\beta$ such that
$z_\beta=0$ for $\beta\leq\alpha$.
Then $\{H_\alpha: \alpha\in I\}$ is a basis of identity neighborhoods
of a group topology making $Z$ into a topological abelian group
$E$, see \cite{bour}, Chap. III, \S1.2 Proposition~1.
We note that $E$ is not discrete because no $H_\alpha$ is trivial.
If we put $K_\alpha=\bigoplus_{\beta\leq\alpha}Z_\beta$
with the discrete topology, 
we have an isomorphism of topological groups
\[ (\forall \alpha)\quad  E\cong K_\alpha \oplus H_\alpha,\]
where $K_\alpha\cong E/H_\alpha$ carries the discrete topology.

Also observe that the projective limit of the projective system
of discrete groups
$\{E/H_\beta \to E/H_\alpha, \alpha\le \beta\}$
agrees with $E$  up to natural isomorphism.
This fact shows that $E$ is indeed a prodiscrete and
thus, in particular, a pro-Lie group.
 In detail, this follows from Lemma \ref{1.5} below.

Recall that an F$_\sigma$-set in a topological space is a countable union
of closed sets.

\begin{Lem}\label{1.2}
In $E$, every F$_\sigma$-set is closed.
\end{Lem}
\begin{proof}
Let $C_n$, for $n\in \NN$, be a sequence of closed subsets of $E$ and set
$C=\bigcup_{n\in \NN}C_n$. Let $z\in E\setminus C$.
Then for each $n\in \NN$ we find
an $\alpha_n\in I$ such that
$(z+ H_{\alpha_n})\cap C_n=\emptyset$. Then we find
a $\gamma\in I$ such that $\alpha_n<\gamma$ for all $n\in\NN$. But then
$(z+H_\gamma)\cap C=\emptyset$, and therefore $C$ is closed.
\end{proof}
Spaces with this property are sometimes called 
\emph{pseudo-discrete spaces}
or \emph{P-spaces}.
\begin{Cor}\label{compactsarefinite}
Every compact subset of $E$ is finite.
In particular, for any topological space $X$ the compact-open topology
on the space $C(E,X)$ of continuous maps from $E$ to $X$ is the topology of pointwise convergence induced by $X^E$.
\end{Cor}
\begin{proof}
     By Lemma \ref{1.2} every countable  subset of $E$ is closed and discrete, because
it cannot have any accumulation point.
\end{proof}
\begin{Cor}\label{EIsNotAkGroup}
We have 
$kE=\bigoplus_{\alpha\in I}Z_\alpha$ with the discrete topology.
In particular, $E$ is not a k-group.
\end{Cor}
\begin{proof}
  Every compact subset of $E$ is finite, and therefore $E$ has
  the same compact subsets
 as the discrete group $Z=\bigoplus_{\alpha\in I}Z_\alpha$.
\end{proof}

\begin{Def}\label{added} We 
introduce  the group $A=\prod_{\alpha\in I}K_\alpha$
with the product topology, as an uncountable product of discrete groups
and define
\[\textstyle \phi\colon E\to A, \quad
\phi\big(\sum_{\alpha\in I}z_\alpha\big)=
(g_\alpha)_{\alpha\in I}\text{ where }g_\alpha=
\textstyle\sum_{\beta\leq\alpha}z_\beta.
\]
\end{Def} 

Recall that each factor $K_\alpha$ carries the discrete topology.
Clearly, $\phi$ is injective.

\begin{Lem}
The subset $\phi(E)$ of $A=\prod_{\alpha\in I} K_\alpha$
consists of all
$g=(g_\alpha)_{\alpha\in I}$, where
$g_\alpha =\sum_{\beta\leq\alpha}x_{\alpha\beta}$ such that
$$(\forall\beta\leq\alpha,\alpha')\quad x_{\alpha\beta}=x_{\alpha'\beta}.$$
Accordingly,  $\phi(E)$ is closed in $A$.
\end{Lem}
\begin{proof}
The elements $g$ in $\phi(E)$ satisfy this condition. For the converse,
we have to show that for such a $g$ we have $x_{\alpha\beta}=0$  for almost
all $\beta$. Indeed otherwise we would have infinitely many different $\beta_n$ with
$x_{a_n\beta_n}\ne0$ for
$n\in\NN$ and we could choose a $\gamma>\beta_n$ for all $n$ and
conclude that the sequence
$(x_{\gamma\beta})_{\beta<\gamma}$ has infinitely many nonzero entries.
In particular $g_\gamma$ would not belong to $K_\gamma$, which is a absurd.
\end{proof}

\begin{Lem}\label{1.5} The morphism
$\phi\colon E\to A$ is an isomorphism onto its image.
\end{Lem}
\begin{proof}
For the continuity of $\phi$ we argue that the composition
$p_\gamma\circ \phi$ for each composition with any projection
$p_\gamma\colon A\to K_\gamma$ onto
any factor of the product is continuous, which  is clear.
For the openness of the corestriction $E\to\phi(E)$ we claim that each 
$\phi(H_\gamma)$ is open in $\phi(E)$. Indeed,
$p^{-1}_\gamma(0)\cap \phi(E)=\phi(H_\gamma)$ is open in
 $\phi(E)$, since
$\{0\}$ is open in $K_\gamma$ by the discreteness of $K_\gamma$.
\end{proof}

\begin{Cor}\label{Cor2.8}
 The k-group $A$ has a closed subgroup $\phi(E)$ which is not a k-group.
 In particular, the underlying topological space of $A$ is not a k-space.
\end{Cor}

\begin{proof}
 By Theorem~\ref{NobleTheorem}, $A$ is a k-group, whereas $\phi(E)\cong E$ is
 not a k-group by Corollary~\ref{EIsNotAkGroup}.
 A closed subspace of a k-space
 is again a k-space, see \cite{eng}, Theorem 3.3.25, 
 hence the underlying space of $A$ is not a k-space.
\end{proof}

Now we study the dual $\widehat{E}$. The dual of the discrete abelian group 
$Z=\bigoplus_{\alpha\in I}Z_\alpha$ is the compact group
$\prod_{\alpha\in I}\widehat{Z_\alpha}$.
We remarked in Corollary~\ref{compactsarefinite}  that the 
compact-open topology on $C(E,\TT)$ coincides with the
topology of pointwise convergence.
Therefore we may view
$\widehat E$ as a subgroup of
$\yhwidehat{\bigoplus_{\alpha\in I}Z_\alpha}
  \cong \prod_{\alpha\in I} \widehat{Z_\alpha}$, algebraically
  and topologically. Let $\xi=(\xi_\alpha)_{\alpha\in I}\in
   \prod_{\alpha\in I}\widehat{Z_\alpha}$ be an element of $\widehat E$ 
and let $W\subseteq\TT$ be an identity neighborhood
which contains no nontrivial subgroup of $\TT$.
If $\xi$ is continuous considered as a map
$\bigoplus_{\alpha\in I}Z_\alpha\supseteq E\to\TT$,
then $\xi(H_\beta)\subseteq W$
holds for some $\beta\in I$. Then $\xi(H_\beta)=\{0\}$.
   In view of $E= K_\beta\oplus H_\beta$ we may identify
   $\widehat{K_\beta}$ with that subgroup {\rm Ann}$(H_\beta)$ of
   $\widehat E$ which annihilates    $H_\beta$. 
Hence $\xi\in {\rm Ann}(H_\beta)\cong\widehat{K_\beta}$.
Thus we may summarize:

\begin{Prop}\label{1.6}
Let $E=\bigoplus_{\alpha\in I}Z_\alpha$ with
the topology introduced in {\rm Definition \ref{1.1}}. Then
\[\widehat E=\bigcup_{\alpha\in I}{\rm Ann}(H_\alpha)
  \subseteq \prod_{\alpha\in I}\widehat{Z_\alpha}\]
is dense in  $\prod_{\alpha\in I}\widehat{Z_\alpha}$
 with the product topology, which induces on $\widehat E$ the
topology of pointwise convergence. Each
{\rm Ann}$(H_\alpha)\cong\widehat{K_\alpha}$ 
is compact and carries the topology of pointwise convergence.
Moreover, $\doublewidehat E\cong \bigoplus_{\alpha\in I}Z_\alpha$
is discrete.
\end{Prop} 
For the last claim we recall a generally well accepted fact.
\begin{Lem}
Let $T$ and $F$
be topological groups, with $F$
complete, $D$ a dense subgroup of $T$, and $\psi\colon D\to F$
a morphism.
Then $\psi$ has a unique continuous extension to a morphism
$\bar\psi\colon T\to F$,
\[
\begin{tikzcd}
    D \arrow{r}{\psi}\arrow[hook]{d} & F \\
    T \arrow{ru}{\exists!}[swap]{\bar\psi}
\end{tikzcd}
\]
\end{Lem}  
\begin{proof}
    See \cite{bour}, Chap. III, \S3.3, Proposition 5,
or \cite{strop}, Cor.8.48.
\end{proof}

Accordingly, each character of $\widehat E$, i.e.,
each element of $\doublewidehat{E}$, uniquely
extends to a character of $\prod_{\alpha\in I}\widehat{Z_\alpha}$,
i.e. to an element of $\bigoplus_{\alpha\in I} Z_\alpha$.
Thus  from $\widehat E\to\prod_{\alpha\in I}\widehat{Z_\alpha}$
we obtain a bijective morphism of abelian topological groups
$\prod_{\alpha\in I}Z_\alpha\to\doublewidehat{E}$.
At this point we note the following.

\begin{Lem}\label{doublehatEdiscrete}
$\doublewidehat{E}$ is discrete and hence
\[{\eta_E}^{-1} \colon \doublewidehat{E}\to E\]
is  a bijective morphism of abelian topological groups.
\end{Lem}
\begin{proof}  Let
 $\Delta\subseteq\prod_{\alpha\in I}\widehat{Z_\alpha}$
denote the set of all $(\xi_\alpha)_{\alpha\in I}$
with at most one $\xi_\alpha\neq0$. The complement of $\Delta$
in $\prod_{\alpha\in I} \widehat{Z_\alpha}$ is
evidently open and therefore $\Delta$ is compact.
Moreover, $\Delta\subseteq\widehat E$,
and $s\Delta\subseteq\Delta$ for every $s\in\ZZ$.
Let $W\subseteq\TT$ be an identity neighborhood which
contains no nontrivial subgroup.
If $z\in \bigoplus_{\alpha\in I}Z_\alpha$ with
$\Delta(z)\subseteq W$, then
$\Delta(sz)=s\Delta(z)\subseteq\Delta(z)\subseteq W$,
and therefore $\Delta(z)=\{0\}$. But then
$z=0$. This shows that
the set $\{0\}=
\{z\in\bigoplus_{\alpha\in I}Z_\alpha: \Delta(z)\in W\}$
is open in $\doublewidehat E$, hence
$\doublewidehat E$ is discrete.
\end{proof}

Let us summarize the features of the group
$E$ we have discussed now!
Recall that $(Z_\alpha,+)$, $\alpha\in I$, 
is a family of nontrivial discrete abelian groups,
with compact dual $\widehat Z_\alpha$,
that $I$ denotes the set of all countable ordinals, and that
$Z=\bigoplus_{\alpha\in I}Z_\alpha$.
Moreover, $K_\alpha$ denotes the discrete group
$K_\alpha=\bigoplus_{\beta\leq\alpha}Z_\beta$,
and $A=\prod_{\alpha\in I}K_\alpha$ with
the product topology. Further,
$\widehat{K_\alpha}\cong \prod_{\beta\le\alpha}\widehat{Z_\beta}$
will be identified with a closed partial product of
$\prod_{\alpha\in I}\widehat{Z_\alpha}\cong
\widehat{\bigoplus_{\alpha\in I}Z_\alpha}$.

\begin{Thm}\label{1.9}
There is a  nondiscrete group topology
on the group $Z=\bigoplus_{\alpha\in I}Z_\alpha$
making it into an abelian topological group $E$
with the following properties.
\begin{enumerate}[\rm(1),nosep]
\item $E$ is isomorphic as a topological group
  to a closed subgroup $F$ of the group $A$.
In particular, $E$ is pro-discrete, pro-Lie and complete.

\item $E$ is not discrete, but every compact subset of $E$ is finite.

\item The character group
  \[\widehat E=\bigcup_{\alpha\in I}\widehat {K_\alpha}
        \subseteq\prod_{\alpha\in I}\widehat{Z_\alpha}\]
is a dense proper subgroup of the compact group
$\prod_{\alpha\in I}\widehat{Z_\alpha}$. In particular,
it is incomplete.
  
\item Its bidual
\[\doublewidehat{E}\cong \bigoplus_{\alpha\in I} Z_\alpha\]
is discrete.

\item The evaluation morphism $\eta_E\colon E\to\doublewidehat{E}$
is bijective, open,
 discontinuous, and is (trivially) continuous on every compact subset
 of $E$.
\item The group $A$ is a k-group, but the closed subgroup
  $F\subseteq A$, $F\cong E$, is not a k-group.
 \end{enumerate}
  \end{Thm}
  \begin{proof}
   Claim (1) and (6) were shown in Lemma~\ref{1.5} and Corollary~\ref{Cor2.8},
   and Claim (2) was shown in Corollary~\ref{compactsarefinite}. Claim (3)
   was shown in Proposition~\ref{1.6}, and Claim (4) and (5) were shown in
   Lemma~\ref{doublehatEdiscrete}. 
  \end{proof}

\begin{Rem}\label{kgroupremark}
The underlying topological space of the group $A$
above is not a k-space,
since otherwise $E$ would also be a k-space, as shown in \cite{eng} Theorem 3.3.25.
On the other hand, $A$ is a k-group by Theorem~\ref{NobleTheorem}.
It follows for the topology $\cT=k\cT$ of $A$ that $\cT\subsetneq\cT_{\max}$,
and that $\cT_{\max}$ is not a group topology for $A$.
Thus, k-groups are not necessarily k-spaces, and closed
subgroups of k-groups are not necessarily k-groups.
We note also that the weakly complete vector spaces $\RR^I$
are k-groups by Theorem~\ref{NobleTheorem}, but that the underlying topological space
$\RR^I$ is not a k-space for our  uncountable set $I$, as is noted in
Kelley's book~\cite{kelley} p.~240, Exercise J(b).
\end{Rem}
  
\begin{Rem}
  In particular, the example  $E$ and its dual $\widehat E$ show
that \emph{the dual of an abelian pro-Lie group may be incomplete}.
We recall that a group which satisfies Pontryagin duality is
called \emph{reflexive}.
Kaplan \cite{kaplan} proves that a product of reflexive groups
is again reflexive.
This applies in particular to the group $A$, which is a product of discrete
abelian groups. Thus $F\subseteq A$ is also an example
of \emph{a closed subgroup of a reflexive group which is not reflexive}.
\textsc{Noble} shows in \cite{nob}, Corollary 3.5 that a closed subgroup
of a \emph{countable} product of locally compact groups satisfies
Pontryagin duality. In our case, the ambient group $A$ is an 
uncountable product of discrete groups. If we put $Z_\alpha=\ZZ$
for all $\alpha$, then $A$ and $E$ are
torsion free. If we put $Z_\alpha=\ZZ(p)$ for all $\alpha$,
then $A$ and $E$ have exponent $p$.
\end{Rem}
As we mentioned above, the discovery of the group $E$ goes back
to \textsc{Leptin} in
\cite{lep}. It is mentioned in the literature repeatedly,
e.g.~by \textsc{Noble} in \cite{nobthesis,nob}, 
and by \textsc{Banaszczyk} in \cite{ban}.

\bigskip

Now we can complement some of the results in Section~1.
\begin{Ex}[Failures of limit preservation]
  We can now list some examples of limits that are not
  preserved by $\iota\colon \kab\to\tab$.
All our examples are based on the group $E$ in Theorem~\ref{1.9}
and the closed injective morphism $\phi\colon E\to A$.
    \begin{enumerate}
        \item \emph{Preservation of projective limits fails.}
          The group $E$ is (in $\tab$) a projective limit of
          the discrete groups $E/H_\alpha\cong K_\alpha$.
          These groups $K_\alpha$ are k-groups,
           but $E$ is not a k-group.
        \item \emph{Preservation of equalizers fails.}
          The morphism $\phi\colon E\to A$ is the
          $\tab$-equalizer of the diagram
        \[A\begin{tikzcd}\arrow[shift left]{r}{p}\arrow[shift
         right,swap]{r}{{\rm const}} &A/\phi(E),\end{tikzcd}
       \quad p(g)=g+\phi(E),\] and both $A$ and
         $A/\phi(E)$ are k-groups, but $E$ is not.
         \item \emph{Preservation of intersections fails.}
         Put \[D=\{(\phi(g),-\phi(g)):g\in E\}\subseteq A\times A
           \quad\text{ and }\quad B=(A\times A)/D.\]
         In $B$ we have the closed subgroups
         $P=(A\times \phi(E))/D\cong A$ and
         $Q=(\phi(E)\times A)/D\cong A$.
         Then $B,P,Q$ are k-groups, but
         $P\cap Q=(\phi(E)\times\phi(E))/D\cong E$ is not.
    \end{enumerate}
    We note that a functor that preserves products and equalizers
    (or products and intersections) preserves all limits.
\end{Ex}

\section{Abelian k-groups}

We consider now the full subcategory $\kab$ of $\tab$
of abelian k-groups in the category of topological abelian groups.
Since \cite{hofii} provides a good deal of information on abelian
pro-Lie groups we begin with the category of abelian pro-Lie groups.
We recall that the topological group $E$ is an abelian pro-Lie group
which fails to be a k-group.
According to \cite{hofii}, Proposition 4.40, the group homomorphism
$\eta_G\colon G\to\doublewidehat{G}$ is injective for
abelian pro-Lie groups. Moreover we have the following.

\begin{Rem}\label{4.1}
For every topological abelian group $G$, the group morphism
\[\eta_G\colon G\to \doublewidehat G\]
is k-continuous.
\end{Rem}
\begin{proof}
See the proof in \cite{hofi}, Theorem 7.7(iii) or \cite{adgb}
Proposition 13.4.1. 
\end{proof}
Thus $\eta_G$ is an injective k-continuous group homomorphism if $G$ is
an abelian pro-Lie group.
From \textsc{Aussenhofer}'s fundamental source \cite{aus} Corollary 21.5, however, we know
much more accurately that $\eta_G$ is bijective and open.
Aussenhofer's result is proved in the context of {\it nuclear groups}.
\footnote{The referee kindly pointed out that there is the following
variant of Aussenhofer's proof.
By Kaplan \cite{kaplan2} Theorem 1 and Theorem 2,
a closed subgroup of a product of locally compact groups
is dual-closed and dual-embedded.
The first part of the proof of Theorem 3.1 in Noble \cite{nob}
then shows that the group homomorphism $\eta_G$ is surjective and open
if $G$ is an abelian pro-Lie group.}
We offer here a different proof, using some categorical framework, 
which (hopefully) illuminates the fact that ${\eta_G}^{-1}$, rather than
$\eta_G$, is continuous.

\begin{Prop}\label{2.1}
For all abelian pro-Lie groups $G$
the morphism of abelian groups $\eta_G\colon G\to \doublewidehat{G}$
is bijective. Its inverse ${\eta_G}^{-1}\colon \doublewidehat{G}\to G$
is  continuous, i.e.,  is
a morphism in the category  $\tab$ of topological abelian groups.
\end{Prop}
Theorem~\ref{1.9}(5) shows that this result cannot be improved.
\begin{proof}
Assume that in the category $\tab$ of topological abelian
groups we have
$G=\lim_{j\in J} G_j$ for a projective  system $\{G_j: j\in J\}$
of Lie groups as in \cite{hofii}, p. 81. Such a presentation is possible
since $G$ is a pro-Lie group. For each $j\in I$
let $p_j\colon G\to G_j$ denote the  morphism in the system.
Since $G_j$ is an abelian Lie group we record that
\[\eta_{G_j}\colon G_j\to \doublewidehat{{G_j}}\]
is an \emph{isomorphism} for each $j\in I$. (See e.g.~\cite{hofi}, 
Theorem 7.63.)
Hence for each $j\in J$ we have a morphism 
\[\doublewidehat{G}\xrightarrow{\ {\doublewidehat{p_j}}\ } \doublewidehat{G_j}
\xrightarrow[\cong]{\ \eta_{G_j}^{-1}\ } G_j,\]
and by naturality of Pontryagin duality for the $G_i$, these maps form
a projective system.
By the limit property of $G=\lim_{j\in I}G_j$,
we have a unique morphism of topological abelian groups
$\eta_G^!\colon \doublewidehat{G}\to G$ such that, for each $j\in I$,
we have the commutative diagram
\[
\begin{tikzcd}
    \doublewidehat{G} \arrow{d}[swap]{\doublewidehat{{p_j}}} \arrow{r}{\eta_G^!} & G \arrow{d}{p_j}\\
    \doublewidehat{G_j} \arrow{r}{\eta_{G_j}^{-1}} &G_j,
\end{tikzcd}
 \qquad j\in J.
\]
We denote the forgetful functor
from the category $\tab$ of topological
abelian groups to the category $\ab$ of abelian groups by $G\mapsto UG$.
In the category of abelian groups $\ab$, 
    for each $j\in I$ we have the commutative diagram
\[
\begin{tikzcd}
UG \arrow{r}{U\eta_G}\arrow{d}[swap]{Up_j}&U\doublewidehat{G}\arrow{d}{U\doublewidehat{p_j}} \\
    UG_j\arrow{r}{U\eta_{G_j}}& U\doublewidehat{G_j}, 
\end{tikzcd}
\qquad j\in J,
\]
  by the naturality of $\eta$.
Staying in the category of abelian groups we combine these two diagrams
and obtain
\[
\begin{tikzcd}
    UG\arrow{r}{U\eta_G}\arrow{d}[swap]{Up_j} & U\doublewidehat{G} \arrow{r}{U\eta_G^!}\arrow{d}{U\doublewidehat{p_j}} &UG \arrow{d}{Up_j} \\
UG_j\arrow{r}{U\eta_{G_j}}&U\doublewidehat{G_j} \arrow{r}{U\eta_{G_j}^{-1}} &UG_j,
\end{tikzcd}
\qquad j\in J.\]
Accordingly,  
since we have $\eta_{G_j}^{-1}\circ \eta_{G_j}=\id_{G_j}$
for all $j\in J$, we have a commutative diagram of
 abelian groups
\[
\begin{tikzcd}
UG\arrow{rr}{U(\eta_G^!\circ \eta_G)} \arrow{d}[swap]{Up_j} &&UG \arrow{d}{Up_j}\\
    UG_j\arrow[equal]{rr} && UG_j,
    \end{tikzcd} 
    \qquad j\in J.\]
The grounding functor $U$ of the category of topological abelian
groups to the category of abelian groups is right adjoint
to the functor 
which attaches to an abelian group the discrete
abelian topological group it supports. Right adjoint functors
preserve limits (see e.g.~\cite{hofi}, Theorem A3.52). 
Thus  $G=\lim_{j\in J} G_j$ implies that
$UG=\lim_{j\in J}UG_j$ holds  in the category of
 abelian groups. Since the fill-in morphism of limits is
 unique, we conclude 
 that \[U({\eta_G^!}\circ \eta_G)=\id_G.\eqno{(*)}\]
 
 However, at this point we need to recall a consequence of
 \textsc{Aussenhofer}'s result  \cite{aus} Corollary 21.5., 
 namely,  that
 for abelian pro-Lie groups, $U\eta_G$ is \emph{surjective}.
 Then $(*)$ implies  that 
 $U{\eta_G^!}=(U\eta_G)^{-1}$, i.e., 
 that $U\eta_G$ is invertible
in the category of abelian groups and so that it is bijective,
  and that
its inverse ${\eta_G}^{-1}=\eta_G^!$ is a morphism in the category of topological
abelian groups, as asserted.
\end{proof}

Again we emphasize that it is illustrated
by the example of the abelian pro-Lie group $E$
of the preceding section  that $\eta_E$ itself
may fail to be continuous.

The existence of the $\tab$-morphism
${\eta_G}^{-1}\colon \doublewidehat{G}\to G$
of (2.1) allows us to pass to the duals to obtain
the morphism
\[\widehat{{\eta_G}^{-1}}\colon\widehat G\to\widehat{\doublewidehat{G}},\qquad
(\forall \chi\in\widehat G,\,\omega\in{\doublewidehat{G}})\
\widehat{{\eta_G}^{-1}}(\chi)\cdot\omega=\omega(\chi).\eqno(1)\]
The topological abelian group $H=\widehat G$ has its own evaluation
morphism $\eta_H$ which by its very definition  is given
by
\[\eta_H\colon H\to\doublewidehat H,\quad
(\forall \chi\in H, \omega\in \widehat H)\
\eta_H(\chi)(\omega)=\omega(\chi). \eqno(2)\]
Comparing statements (1) and (2) we conclude the following

\begin{Cor}\label{2.2}
For any topological abelian group $G$ for which
$\eta_G$ has an inverse ${\eta_G}^{-1}$ which is a morphism of topological abelian
groups,
\[\widehat{{\eta_G}^{-1}}\colon\widehat G\to\widehat{\doublewidehat{G}}
\qquad\text{ equals }\qquad
\eta_{\widehat G}\colon \widehat G\to \widehat{\doublewidehat{G}}.\]
\end{Cor}
This corollary allows us to formulate and prove the following result.

\begin{Thm}\label{pontr}
For an abelian 
pro-Lie group $G$ the following statements are equivalent.

\begin{enumerate}[\rm(1),nosep]
    \item $G$ is the character group of
an abelian topological group $H$ for which $\eta_H$ is bijective and open.
\item $\eta_G\colon G\to \doublewidehat G$
is  an  isomorphism of abelian pro-Lie groups. 
\end{enumerate}
\end{Thm}
\begin{proof}
(1) implies (2): Since ${\eta_H}^{-1}\colon  \doublewidehat{H}\to H$
is a morphism by hypothesis,
so is its dual $\widehat{{\eta_H}^{-1}}\colon \widehat H\to\widehat{\doublewidehat{H}}$. It agrees 
with  the  morphism $\eta_G\colon G\to\doublewidehat{G}$ by Corollary \ref{2.2}.

(2) implies (1): Condition (2) allows us to apply Corollary \ref{2.2}, to
set $H= \widehat G$,  and to conclude that $\widehat H \cong \doublewidehat G\cong G$. 
 \end{proof}

 In Theorem~\ref{1.9} above 
 we  saw the example of a nondiscrete prodiscrete  topological abelian
 group $E$ on the underlying group
 $\bigoplus_{\alpha\in I}Z_\alpha$, which showed 
 that $\eta_E$ can fail to be continuous in general, and so,
 accordingly, that ${\eta_E}^{-1}$ can fail to be open.
 
 That example also illustrates
 $\doublewidehat E\cong \bigoplus_{\alpha\in I}Z_\alpha$ with the
 discrete topology and
 $\widehat{\doublewidehat{E}}\cong\prod_{\alpha\in I}\widehat{Z_\alpha}$.

The subcategory of all real topological vector spaces $G$
which are pro-Lie groups is  the category of
\emph{weakly complete real vector spaces}, whose
 dual category is the category
of all real topological vector spaces $\widehat G$ 
 endowed with the finest possible
vector space topology, as is shown in \cite{hofi}, Appendix 7,
pp.\ 932ff., or \cite{hofii}, Appendix 3, pp.\ 737ff.,
and such
topological vector spaces are pro-Lie groups only as long as they
are finite dimensional.  
Accordingly,
by way of example,
if $G=\RR^{(J)}$ for an infinite set $J$ (e.g.\ $J=\NN$), then
$G$, equipped with the finest locally convex topology,
is the dual of the abelian pro-Lie group $\RR^J$ 
(see \cite{hofi}, p.\-932ff., \cite{hofii}, p.\-737ff.)
illustrating Theorem  \ref{pontr} and showing that the category of 
abelian
pro-Lie groups fails to be closed under passage to the duals.

A core result on any abelian pro-Lie group $G$
reads as follows (see \cite{hofii}, Theorem 4.22 on pp.~144, 145):

\begin{Thm} Every abelian pro-Lie group $G$ is isomorphic to
  $V\times H$ where $V$ is a weakly complete real vector group
  isomorphic to $\RR^J$ for some set $J$, and the idenity component $H_0$ 
  of $H$ is  compact,
  and the unique union $C={\rm comp}(G)$ of all compact subgroups of $G$ is 
  a closed subgroup of $H$. Moreover, $H/C$ does not contain any nontrivial
  compact subgroup.
\end{Thm}
  We know that $kV=V$ by Theorem~\ref{NobleTheorem}. On the other hand, we
  have the example of the group $E$ from Section 2. If every
  element of the $Z_\alpha$ has finite order, then ${\rm comp}(E)=E$. and $kE\neq E$.
  Hence $kC\neq C$ in general.

\bigskip

From Remark~\ref{4.1} and Proposition~\ref{2.1}
we get immediately a result stated in the 
introduction.
\begin{Prop}\label{kProLie}
    If $G$ is an abelian pro-Lie group which is also a $k$-group,
    then there is an isomorphism of topological groups
    \[
    G\xrightarrow[\ \cong\ ]{\eta_G}\doublewidehat{G}.
    \]
    In particular, $\doublewidehat G$ is then a k-group.
\end{Prop}
Let $G$ be a topological abelian group.
If $\eta_G\colon G\to\doublewidehat{G}$ has a continuous inverse, as is the case for abelian
pro-Lie groups by Proposition \ref{2.1}, we have also the $\kab$-morphism
$k({\eta_G}^{-1}):k\doublewidehat G\to kG$. Accordingly, we get 

\begin{Prop}\label{4.2}  
For an abelian  {\rm pro-Lie} group $G$,
  there is an isomorphism
\[kG\xrightarrow[\cong]{k\eta_G} k\doublewidehat G\]
  inside $\kab$ and there is a commutative
  diagram
  \[
    \begin{tikzcd} kG \arrow{d}[swap]{\kappa_G} &
\arrow{l}{\cong}[swap]{k{\eta_G}^{-1}}k\doublewidehat G
\arrow{d}{\kappa_{\doublewidehat G}} \\
G& \arrow{l}[swap]{{\eta_G}^{-1}}\doublewidehat G
  \end{tikzcd} 
  \]
of bijective $\tab$-morphisms. 
 \end{Prop}

Quite generally, we have 
the morphism $\widehat{\kappa_G}\colon \widehat G\to\widehat{kG}$
and its dual 
$\doublewidehat{\kappa_G}\colon\doublewidehat{kG}\to \doublewidehat G$.

We note that there are 3 commutative square diagrams of
$\kab$-morphisms as follows:

Firstly,
\[
\begin{tikzcd}
    kG \arrow{r}{\eta_{kG}}\arrow{d}[swap]{\kappa_G}
       &\doublewidehat{kG}
    \arrow{d}{\doublewidehat{\kappa_G}}\\ 
    G\arrow[dashed]{r}{(\eta_G)} & \doublewidehat G.
\end{tikzcd}
\]

Secondly,
\[
\begin{tikzcd}
    kG \arrow{r}{k\eta_G}\arrow{d}[swap]{\kappa_G} 
        &k\doublewidehat{G} 
    \arrow{d}{\kappa_{\doublewidehat{G}}}\\
    G \arrow[dashed]{r}{(\eta_G)} & \doublewidehat{G}.
\end{tikzcd}
\]

Thirdly,
\[
\begin{tikzcd}
    kG \arrow{r}{\eta_{kG}}\arrow{d}[swap]{k\eta_G} 
        &\doublewidehat{kG} 
    \arrow{d}{\doublewidehat{\kappa_G}}\\
k\doublewidehat{G}\arrow{r}{\kappa_{\doublewidehat{G}}}&\doublewidehat{G}.
\end{tikzcd}
\]

 In these commutative diagrams, the dotted arrows may not be continuous,
but all arrows are morphisms in the category $\ab$.
 All of them have one and the same diagonal morphism

\[d_G=\doublewidehat{\kappa_G}\circ\eta_{kG}
  =\kappa_{\doublewidehat G}\circ k\eta_G
  =\eta_G\circ \kappa_G.\]

By Proposition 3.6 the morphism $\eta_{kG}$ factors through
   $\kappa_{\doublewidehat{kG}}\colon k\doublewidehat{kG}\to
   \doublewidehat{kG}$, so that
   $\eta_{kG}=\kappa_{\doublewidehat{kG}}\circ {\eta_{kG}}$
   and we obtain
$d_G=\doublewidehat{\kappa_G}\circ \kappa_{\doublewidehat{kG}}
\circ {\eta_{kG}}'$.   

\bigskip

The morphism $\eta_G$ in the category of abelian groups
$\ab$ is k-continuous by Remark~\ref{4.1}
and is bijective and open if $G$ is an abelian pro-Lie group
by Proposition~\ref{2.1}. Since both $\kappa_G$ and $\eta_G$
are bijective  in this case, we obtain

\begin{Rem} \label{Rem 4.4} For every abelian pro-Lie group $G$,
  the morphism $d_G\colon kG\to\doublewidehat G$ is bijective.
\end{Rem}

\medskip

This discussion  has led to various representations of the significant
morphism $d_G= kG\to\doublewidehat G$. Chances are that it is
an isomorphism for abelian pro-Lie groups.
 Since we know from Proposition  \ref{4.2} that $k\eta_G$
is an isomorphism for abelian pro-Lie groups, we record the following
observations:

\begin{Cor} \label{final} For an abelian pro-Lie group $G$, the
  following statements are equivalent:

  {\rm(i)} $\doublewidehat G$ is a k-group.

  {\rm(ii)} $\kappa_{\doublewidehat G}$ is an isomorphism.

  {\rm(iii)} $d_G$ is an open morphism.

  {\rm(iv)} $d_G$ is an isomorphism of topological groups.  
\end{Cor}
  
The Example $E$ of Theorem \ref{1.9} shows that these statements will
not imply that $\eta_G\colon G\to \doublewidehat G$ is an isomorphism.

\medskip

These aspects encourage us to ask the following question, which may be
an indication that the duality of abelian topological groups is still
a source of challenges:

\medskip
\noindent
{\bf Question.} \quad {\it Is the double dual
$\doublewidehat G =\Hom(\Hom(G,\RR/\ZZ),\RR/\ZZ)$ of an
 abelian pro-Lie group $G$ always a k-group?}

\end{document}